\documentclass{amsart}

\usepackage{amsmath,amssymb,amsthm}
\usepackage{algorithm}
\usepackage{algpseudocode}

\newtheorem{theorem}{Theorem}[section]

\newtheorem{lemma}[theorem]{Lemma}
\newtheorem{proposition}[theorem]{Proposition}

\theoremstyle{definition}

\theoremstyle{remark}
\newtheorem{remark}{Remark}
\newcommand{\F}{\mathbb{F}}

\newcommand{\ord}{\mathrm{ord}}
\newcommand{\lcm}{\mathrm{lcm}}

\newcommand{\R}{\mathcal{R}}

\newcommand{\Li}{\mathcal{L}}
\newcommand{\T}{\mathcal{T}}
\newcommand{\X}{\mathcal{X}}
\newcommand{\Y}{\mathcal{Y}}

\newcommand{\N}{\mathcal{N}}

\begin{document}

\title[The line property for 2-primitive elements in $\F_{q^2}/\F_q$]{The translate and line properties for 2-primitive elements in quadratic
extensions}

\author{Stephen D. Cohen}
\address{6 Bracken Road, Portlethen, Aberdeen AB12 4TA, Scotland, UK}
\email{Stephen.Cohen@glasgow.ac.uk}

\author{Giorgos Kapetanakis}
\address{Department of Mathematics and Applied Mathematics, University of Crete, Voutes Campus, 70013 Heraklion, Greece}
\email{gnkapet@gmail.com}

\date{\today}
\thanks{The first author is Emeritus Professor of Number Theory, University of Glasgow}


\begin{abstract}
Let $r,n>1$ be integers and $q$ be any prime power $q$ such that $r\mid q^n-1$. We say that the extension $\F_{q^n}/\F_q$ possesses the line property for $r$-primitive elements if, for every $\alpha,\theta\in\F_{q^n}^*$, such that $\F_{q^n}=\F_q(\theta)$, there exists some $x\in\F_q$, such that $\alpha(\theta+x)$ has multiplicative order $(q^n-1)/r$.   Likewise, if, in the above definition, $\alpha$ is restricted to the value $1$, we  say that  $\F_{q^n}/\F_q$ possesses the translate property.   In this paper we take  $r=n=2$  (so that  necessarily $q$ is odd) and prove  that  $\F_{q^2} /\F_q$ possesses the translate  property for 2-primitive elements unless $q \in \{5,7,11,13,31,41\}$.  With some additional theoretical and computational effort, we show also that $\F_{q^2} /\F_q$  possesses the line  property for 2-primitive elements  unless  $q \in \{3,5,7,9,11,13,31,41\}$.
\end{abstract}

\maketitle

%\date{\today}

\section{Introduction}
Let $q$ be a prime power and $n\geq 2$ an integer. We denote by $\F_q$ the finite field of $q$ elements and by $\F_{q^n}$ its extension of degree $n$.
It is well-known that the multiplicative group $\F_{q^n}^*$ is cyclic; its  generators are called  \emph{primitive elements}. The theoretical importance of primitive elements is complemented by their numerous applications in practical areas such as cryptography.

In addition to their theoretical interest, elements of $\F_{q^n}^*$ that have high order, without necessarily being primitive, are of great practical interest because in several applications they may replace primitive elements.   Accordingly, recently researchers have worked on the the effective construction of such high order elements, \cite{gao99,martinezreis16,popovych13}, since that  of  primitive elements themselves remains an open problem.

With that in mind, we call an element of order $(q^n-1)/r$, where $r\mid q^n-1$, \emph{$r$-primitive}, i.e., the primitive elements are exactly the $1$-primitive elements. In this line of work, the existence of $2$-primitive elements that also possess other desirable properties has been recently considered \cite{cohenkapetanakis19,kapetanakisreis18}.

We call some $\theta\in\F_{q^n}$ a \emph{generator} of the extension $\F_{q^n}/\F_q$ if $\F_{q^n} = \F_q(\theta)$ and, if $\theta$ is a generator of $\F_{q^n}/\F_q$, we call the set \[ \T_\theta := \{ \theta +x \, :\, x\in\F_q\} \] the \emph{set of translates} of $\theta$ over $\F_q$ and every element of this set a \emph{translate} of $\theta$ over $\F_q$.
%It is not hard to see that the set of generators of an extension is in fact partitioned into a number of distinct sets of translates.
We say that an extension $\F_{q^n} / \F_q$ possesses the \emph{translate property for $r$-primitive elements}, if every set of translates contains an $r$-primitive element. In particular, for $r=1$ we simply call it the \emph{translate property}.
A classical result in the study of primitive elements is the following.% \cites{carlitz53a,davenport37}.
\begin{theorem}[Carlitz-Davenport]\label{thm:ca-da}
Let $n$ be an integer. There exist some $T_1(n)$ such that for every prime power $q> T_1(n)$, the extension $\F_{q^n}/\F_q$ possesses the translate property.
\end{theorem}
The above was first proved by Davenport \cite{davenport37}, for prime $q$, while Carlitz \cite{carlitz53a} extended it to the stated form. Interest in this problem was renewed by recent applications of the translate property in semifield primitivity, \cite{kapetanakislavrauw18,rua15,rua17}.

Let $\theta$ be a generator of the extension $\F_{q^n}/\F_q$ and take some $\alpha\in\F_{q^n}^*$. We call the set \[ \Li_{\alpha,\theta} := \{ \alpha(\theta+x) \, : \, x\in\F_q \} \] the \emph{line} of $\alpha$ and $\theta$ over $\F_q$. An extension $\F_{q^n}/\F_q$ is said to possess the \emph{line property for $r$-primitive elements} if every line of this extension contains an $r$-primitive element. When $r=1$, we refer to this property as the \emph{line property}.
A natural generalization of Theorem~\ref{thm:ca-da} is the following, \cite[Corollary~2.4]{cohen10}.
\begin{theorem}[Cohen] \label{thm:line1}
Let $n$ be an integer. There exist some $L_1(n)$ such that for every prime power $q> L_1(n)$, the extension $\F_{q^n}/\F_q$ possesses the line property.
\end{theorem}
The authors have recently \cite{cohenkapetanakis19b} established the following extension of Theorems~\ref{thm:ca-da} and \ref{thm:line1} to $r$-primitive elements.
\begin{theorem}\label{thm:our_ca-da}
Let $n$ and $r$ be integers. There exist some $L_r(n)$ such that for every prime power $q> L_r(n)$, with the property $r\mid q^n-1$, the extension $\F_{q^n}/\F_q$ possesses the line property for $r$-primitive elements. If we confine ourselves to the translate property for $r$-primitive elements, the same is true for some $T_r(n)\leq L_r(n)$.
\end{theorem}
A natural, but apparently challenging, related question is identifying the exact value of the numbers $T_1(n)$ and $L_1(n)$ for given $n$.  Indeed, only  a handful of these are known.  In particular, the first author, in \cite{cohen83}, proved that $T_1(2) = L_1(2) = 1$ and, in \cite{cohen09},  that $T_1(3) = 37$. Bailey et al.~\cite{baileycohensutherlandtrudgian19} proved that $L_1(3) = 37$ and estimated $T_1(4) \leq L_1(4) \leq 102829$.

In this paper we consider the case in which $r=n=2$ 
and establish complete existence results by proving the following theorems.
\begin{theorem}\label{thm:trans2}
For every odd prime power $q\neq 5$, $7$, $11$, $13$, $31$ or $41$ the extension $\F_{q^2}/\F_q$ possesses the translate property for $2$-primitive elements. In particular, $T_2(2) = 41$.
\end{theorem}
\begin{theorem}\label{thm:line}
For every odd prime power $q\neq 3$, $5$, $7$, $9$, $11$, $13$, $31$ or $41$, $\alpha\in\F_{q^2}^*$ and $\theta\in\F_{q^2}\setminus \F_q$, there exists some $x\in\F_q$ such that $\alpha(\theta+x)$ is $2$-primitive. In particular, $L_2(2) = 41$.
\end{theorem}

The above results are the fruits of combined theoretical and computational methods. Namely, first, by proving Theorem~\ref{thm:main_n=r=2}, we effectively estimate $L_2(2)$ and $T_2(2)$  theoretically by a sieving method, cf.~\cite{cohenhuczynska03}.    This  leaves a small number (around 100) of extensions unresolved, the largest prime power remaining being $q =3541$.    Then, we employ computational methods (extensive as regards the line property) to deal with the remaining extensions.
\section{Preliminaries}\label{sec:preliminaries}
%
%
%\subsection{Characters and character sums}
%
We begin by introducing the notion of freeness. Let $m\mid q^n-1$, an element $\xi \in \F_{q^n}^*$ is \emph{$m$-free} if $\xi = \zeta^d$ for some $d\mid m$ and $\zeta\in\F_{q^n}^*$ implies $d=1$. It is clear that primitive elements are exactly those that are $q_0$-free, where $q_0$ is the square-free part of $q^n-1$. It is also evident that there is some relation between $m$-freeness and multiplicative order.
\begin{lemma}[\cite{huczynskamullenpanariothomson13}, Proposition~5.3]\label{lemma:m-free}
If $m\mid q^n-1$ then $\xi\in\F_{q^n}^*$ is $m$-free if and only if $\gcd\left( m,\frac{q^n-1}{\ord\xi} \right)=1$.
\end{lemma}

Throughout this work, a \emph{character} is a multiplicative character of $\F_{q^n}^*$, while we denote by $\chi_0$ the trivial multiplicative character.
Vinogradov's formula yields an expression for the characteristic function of $m$-free elements in terms of multiplicative characters, namely:
\begin{equation}\label{eq:omega}
\Omega_m(x) := \theta(m) \sum_{d\mid m} \frac{\mu(d)}{\phi(d)} \sum_{\ord\chi = d} \chi(x) ,
\end{equation}
where $\mu$ stands for the M\"{o}bius function, $\phi$ for the Euler function, $\theta(m) := \phi(m)/m$ and the inner sum suns through multiplicative characters of order $d$. Furthermore, a direct consequence of the orthogonality relations is that the characteristic function for the elements of $\F_{q^n}^*$ that are $k$-th powers, where $k\mid q^n-1$, can be written as
\begin{equation}\label{eq:w}
w_k (x) := \frac{1}{k} \sum_{d\mid k} \sum_{\ord\chi=d} \chi(x) .
\end{equation}
We will use character sums to establish our results. For the following, see \cite[Lemma~3.3]{cohen10}.
\begin{lemma}[Cohen]\label{lem:cohen}
Let $\theta\in\F_{q^2}$ be such that $\F_{q^2} = \F_q(\theta)$ and $\chi$ a non-trivial character. Set
\[ B := \sum_{x\in\F_q} \chi (\theta+x) . \]
\begin{enumerate}
\item If $\ord\chi\nmid q+1$, then $|B| = \sqrt{q}$.
\item If $\ord\chi\mid q+1$, then $B = -1$.
\end{enumerate}
\end{lemma}
Furthermore, let $W(R)$ be the number of the square-free divisors of $R$. The following provides an efficient bound for this function.
\begin{lemma}\label{lem:w(r)}
Let $R, a$ be positive integers and let $p_1, \ldots, p_j$ be the distinct prime divisors of $R$ such that $p_i\le 2^a$. Then $W(R)\le c_{R, a}R^{1/a}$, where
\[ c_{R, a}=\frac{2^j}{(p_1\cdots p_j)^{1/a}}. \]
In particular,
$d_R:=c_{R, 8}<4514.7$ for every $R$.
\end{lemma}
\begin{proof}
The statement is an immediate generalization of \cite[Lemma~3.3]{cohenhuczynska03} and can be proved using multiplicativity. The bound 
 for $d_R$ can be easily computed.
\qed
\end{proof}
\section{Sufficient conditions}\label{sec:n=r=2}
%
%
%
%
%We progress to an effective version of Theorem~\ref{thm:our_ca-da}, when $r=n=2$.

% \begin{remark}
% For fixed $t$, the number $d_t$ is easily computed and in most cases (especially if $t$ is not large) the actual value of $d_t$ is remarkably smaller than the generic bound.
% \end{remark}
%
Observe that for our case, since $2\mid q^2-1$, we must further assume that $q$ is odd, in which case, evidently, $4\mid q^2-1$.
Now, let $R'$ be the square-free part of the odd part of $q^2-1$ and take some $\R \mid R'$ and let $\Gamma_\R$ be the characteristic function for $\R$-free elements, that are squares, but not fourth powers. First, we will express $\Gamma_\R$ with the help of characters.

With the notation of Section~\ref{sec:preliminaries}, $\Gamma_\R$ can be expressed as
\[
\Gamma_\R(x)  = \Omega_{\R}(x) w_{2}(x)(1-w_{4}(x)) ,
\]
where $x\in\F_{q^2}^*$. Moreover, a fourth power is also a square, hence, for $x\in\F_{q^2}^*$, $w_{2}(x)w_{4}(x) = w_{4}(x)$ and the latter yields
\begin{equation} \label{eq:Omega2}
\Gamma_\R(x) = \Omega_{\R}(x) (w_{2}(x)-w_{4}(x)).
\end{equation}
Furthermore, for every $x\in\F_{q^n}^*$,
\[
w_{2}(x)-w_{4}(x)  = \frac{1}{2} \sum_{\delta\mid 2} \sum_{\ord\chi=\delta} \chi(x) - \frac{1}{4} \sum_{\delta\mid 4} \sum_{\ord\chi=\delta} \chi(x) 
 = \frac{1}{2} \sum_{\delta\mid 4} \sum_{\ord\chi = \delta} \ell_{\delta} \chi(x) ,
\]
where, for $\delta\mid 4$,
\[ 
\ell_{\delta} := \begin{cases} 1/2 , & \text{if } \delta\neq 4 , \\ -1/2 , & \text{if } \delta=4 . \end{cases}
\]
Finally, we insert the above and the expressions \eqref{eq:omega} and \eqref{eq:w} into \eqref{eq:Omega2}, and  obtain
\begin{equation} \label{eq:Omega_expanded}
\Gamma_{\R}(x) = \frac{\theta(\R)}{2} \sum_{\substack{d\mid \R \\ \delta\mid 4}} \frac{\mu(d)}{\phi(d)} \ell_{\delta} \sum_{\substack{\ord\chi=d \\ \ord\psi=\delta}} (\chi\psi)(x) ,
\end{equation}
where $x\in\F_{q^n}^*$ and $(\chi\psi)$ stands for the product of the corresponding characters, itself a character.

Next, fix some $\theta\in\F_{q^2}$ such that $\F_{q^2}=\F_q(\theta)$ and some $\alpha\in\F_{q^2}^*$. Further, let $\N_\R(\theta,\alpha)$ stand for the number of $\R$-free elements, that are squares, but not fourth powers, in the set $\{\alpha(\theta+x)\, :\, x\in\F_q\}$, i.e.,
\[ \N_\R(\theta,\alpha) = \sum_{x\in\F_q} \Gamma_\R (\alpha(\theta+x)) . \]
Clearly, for our purposes, it suffices to show that
$
\N_{R'}(\theta,\alpha)\neq 0
$.

The above expression of $\N_\R(\theta,\alpha)$, combined with \eqref{eq:Omega_expanded}, yield
\begin{equation}\label{eq:N_2(1)}
\frac{\N_\R(\theta,\alpha)}{\theta(\R)} = \frac 12 \sum_{\substack{d\mid \R \\ \delta\mid 4}} \frac{\mu(d)}{\phi(d)} \ell_{\delta} \sum_{\substack{ \ord\chi=d \\ \ord\psi=\delta}} \X_{\alpha,\theta} (\chi,\psi) ,
\end{equation}
where
\[ \X_{\alpha,\theta}(\chi,\psi) := \sum_{x\in\F_q} (\chi\psi)(\alpha(\theta+x)) . \]
Moreover, let $\{\eta_1,\eta_2\}$ be the characters of order $4$; then the characters whose order divides $4$ will be $\{\chi_o,\eta,\eta_1,\eta_2\}$, where $\eta$ is the \emph{quadratic character}, the character of order $2$. With these in mind, we rewrite \eqref{eq:N_2(1)} as follows:
\begin{equation}\label{eq:N_2(2)}
\frac{4\cdot \N_\R(\theta,\alpha)}{\theta(\R)} = \sum_{d\mid\R} \frac{\mu(d)}{\phi(d)} \sum_{\ord\chi=d} \Y_{\alpha,\theta}(\chi) ,
\end{equation}
where
\[
\Y_{\alpha,\theta}(\chi) := \X_{\alpha,\theta}(\chi,\chi_0) + \X_{\alpha,\theta}(\chi,\eta) - \X_{\alpha,\theta}(\chi,\eta_1) - \X_{\alpha,\theta}(\chi,\eta_2) .
\]
Now, we distinguish the cases $q\equiv 1\pmod{4}$ and $q\equiv 3\pmod{4}$.

First, assume $q\equiv 1\pmod{4}$. Then $4\nmid q+1$, hence Lemma~\ref{lem:cohen} implies that
\begin{enumerate}
  \item for $\chi_0$, $\Y_{\alpha,\theta}(\chi_0) = q - A$, where $A= \X_{\alpha,\theta}(\chi_0,\eta) - \X_{\alpha,\theta}(\chi_0,\eta_1) - \X_{\alpha,\theta}(\chi_0,\eta_2)$, that is, $|A|\leq 1+2\sqrt{q}$,
  \item for $1\neq \ord\chi\mid q+1$, $|\Y_{\alpha,\theta}(\chi)| \leq 2+2\sqrt{q}$,
  \item for $\ord\chi\nmid q+1$, $|\Y_{\alpha,\theta}(\chi)| \leq 4\sqrt{q}$.
\end{enumerate}
If we assume that $q\equiv 3\pmod{4}$, then $4\mid q+1$ and Lemma~\ref{lem:cohen} implies that
\begin{enumerate}
  \item for $\chi_0$, $\Y_{\alpha,\theta}(\chi_0) \geq q-3$,
  \item for $1\neq \ord\chi\mid q+1$, $|\Y_{\alpha,\theta}(\chi)| \leq 4$,
  \item for $\ord\chi\nmid q+1$, $|\Y_{\alpha,\theta}(\chi)| \leq 4\sqrt{q}$.
\end{enumerate}

We insert the above in \eqref{eq:N_2(2)} and get the following.
\begin{proposition}\label{prop:n=2(1)} 
Let $q$, $\alpha$, $\theta$ and $\R$ be as above and let $\R_1$ be the product of the prime divisors of $\R$ that divide $q+1$.
\begin{enumerate}
\item If $q\equiv 1\pmod{4}$, then
\begin{equation} \label{eq:n=2_ip(1)}
\frac{4\cdot \N_\R(\theta,\alpha)}{\theta(\R)} \geq q+1 - 4W(\R)\sqrt{q} + 2W(\R_1) (\sqrt{q}-1),
\end{equation}
that is, if
\[ q+1 > 4\left( W(\R)\sqrt{q} - W(\R_1) \left( \frac{\sqrt{q}-1}{2} \right)\right) ,\]
then $\N_\R(\theta,\alpha)\neq 0$.
\item If $q\equiv 3\pmod{4}$, then
\begin{equation} \label{eq:n=2_ip(2)}
\frac{4\cdot \N_\R(\theta,\alpha)}{\theta(\R)} \geq q+1 - 4W(\R) \sqrt{q} + 4W(\R_1)(\sqrt{q}-1),
\end{equation}
that is, if
\[ q+1 > 4(W(\R)\sqrt{q} - W(\R_1)(\sqrt{q}-1)) ,\]
then $\N_\R(\theta,\alpha)\neq 0$.
\end{enumerate} 
\end{proposition}

Our next aim is to relax the conditions of Proposition~\ref{prop:n=2(1)}. For this purpose, we adapt the  sieving techniques of  Cohen-Huczynska, \cite{cohenhuczynska03}.
\begin{proposition}[Sieving inequality]\label{prop:siev0}
Let $m\mid R'$ and $\theta,\alpha\in\F_{q^2}^*$ such that $\F_{q^2} = \F_q(\theta)$. In addition, let $\{r_1,\ldots,r_s\}$ be a set of divisors of $m$ such that $\gcd(r_i,r_j)=r_0$ for every $i\neq j$ and $\lcm(r_1,\ldots ,r_s)=m$. Then
\[
\N_m(\theta,\alpha) \geq \sum_{i=1}^s \N_{r_i}(\theta,\alpha) - (s-1)\N_{r_0}(\theta,\alpha) .
\]
\end{proposition}
\begin{proof}
For any $l\mid R'$, let $S_l$ be the set of $l$-free elements of the form $\alpha(\theta+x)$, where $x\in\F_q$, that are squares, but not fourth powers. In other words, $|S_l| = \N_l(\theta,\alpha)$.
Accordingly, we may work with $|S_l|$ instead of $\N_l(\theta,\alpha)$.

We will use induction on $s$. The result is trivial for $s=1$. For $s=2$ notice that $S_{r_1} \cup S_{r_2} \subseteq S_{r_0}$ and that $S_{r_1} \cap S_{r_2} = S_m$. The result follows after considering the cardinalities of those sets.

Next, assume that our hypothesis holds for some $s\geq 2$. We shall prove our result for $s+1$. Set $r:=\lcm(r_1,\ldots ,r_s)$ and apply the $s=2$ case on $\{r,r_{s+1}\}$. The result follows from the induction hypothesis.
\qed
\end{proof}
Write $R'=kp_1\cdots p_s$, where $p_1$,\ldots,$p_s$ are distinct primes and $\varepsilon := 1-\sum_{i=1}^s 1/p_i$, with $\varepsilon=1$ when $s=0$. Further, suppose that $p_i\mid q+1$ for $i=1,\ldots,r$ and $p_i\nmid q+1$ for $i=r+1,\ldots,s$. Finally, set $\varepsilon' := 1-\sum_{i=1}^r 1/p_i$ and let $k_1$ be the part of $k$, that divides $q+1$.
\begin{proposition}\label{prop:siev1}
Let $q$, $\alpha$, $\theta$ and $R'$ be as above. Additionally, let $\varepsilon$ and $\varepsilon'$ be as above and assume that $\varepsilon>0$.
\begin{enumerate}
\item If $q\equiv 1\pmod{4}$ and
\[ q+1 > 4\left[ W(k)\left( \frac{s-1}{\varepsilon} +2 \right) \sqrt{q} - W(k_1) \left( \frac{r-1+\varepsilon'}{\varepsilon} +1 \right) \left( \frac{\sqrt{q}-1}{2} \right)\right] ,\]
then $\N_{R'}(\theta,\alpha)\neq 0$.
\item If $q\equiv 3\pmod{4}$ and
\[ q+1 > 4\left[  W(k)\left( \frac{s-1}{\varepsilon} +2 \right)\sqrt{q}- W(k_1) \left( \frac{r-1+\varepsilon'}{\varepsilon} +1 \right) (\sqrt{q}-1) \right]  ,\]
then $\N_{R'}(\theta,\alpha)\neq 0$.
\end{enumerate} 
\end{proposition}
\begin{proof}
Proposition~\ref{prop:siev0} implies that
\begin{align}
\N_{R'}(\theta,\alpha) & \geq \sum_{i=1}^s \N_{kp_i}(\theta,\alpha) - (s-1) \N_{k}(\theta,\alpha) \nonumber \\
 & \geq \varepsilon \N_{k}(\theta,\alpha) - \sum_{i=1}^s \left| \N_{kp_i}(\theta,\alpha) - \left( 1-\frac{1}{p_i} \right) \N_{k}(\theta,\alpha) \right| . \label{eq:siev_ip(0)}
\end{align}
Notice that $\theta(kp_i) = \theta(k)(1-1/p_i)$. It follows from \eqref{eq:N_2(2)} that
\begin{equation} \label{eq:siev_ip(1)}
\N_{kp_i}(\theta,\alpha) - \left( 1-\frac{1}{p_i} \right) \N_{k}(\theta,\alpha) =\frac{\theta(k)(p_i-1)}{4p_i} \sum_{d\mid k} \frac{\mu(dp_i)}{\phi(dp_i)} \sum_{\ord\chi=dp_i} \Y_{\alpha,\theta}(\chi) .
\end{equation}

First assume that $q\equiv 1\pmod{4}$. We repeat the arguments that led us to \eqref{eq:n=2_ip(1)} for \eqref{eq:siev_ip(1)}. If $i=1,\ldots,r$, i.e., $p_i\mid q+1$, then
\begin{multline*}
\left| \N_{kp_i}(\theta,\alpha) - \left( 1-\frac{1}{p_i} \right) \N_{k}(\theta,\alpha) \right| \leq \\ \theta(k) \left( 1-\frac{1}{p_i} \right) \big[ 2\sqrt{q} (W(k)-W(k_1)) + (1+\sqrt{q}) W(k_1) \big] ,
\end{multline*}
since $W(kp_i) = 2W(k)$ and $W(k_1p_i) = 2W(k_1)$. Similarly, if $i=r+1,\ldots,s$, i.e., $p_i\nmid q+1$, then
\[
\left| \N_{kp_i}(\theta,\alpha) - \left( 1-\frac{1}{p_i} \right) \N_{k}(\theta,\alpha) \right| \leq \theta(k) \left( 1-\frac{1}{p_i} \right) 2\sqrt{q} W(k).
\]
The combination of \eqref{eq:n=2_ip(1)}, \eqref{eq:siev_ip(0)}, \eqref{eq:siev_ip(1)} and the above bounds yields the desired result.

The case when  $q\equiv 3\pmod{4}$  follows in the same fashion, but with \eqref{eq:n=2_ip(2)} in mind.
\qed
\end{proof}
We are now ready to proceed with the numerical aspects.
\section{Numerical aspects}\label{subsec:numerical}
All the mentioned computations and algorithms were implemented with the \textsc{SageMath} software. Since, in some cases, finding a computationally efficient or viable way to perform our calculations was non-trivial, the basic steps of our calculations are described in detail. Furthermore, we note that a modern mid-range laptop can perform the computations of this subsection in less than two minutes.

We start with the simplest sufficient condition to check.  This  derives from (\ref{eq:n=2_ip(1)}) and (\ref{eq:n=2_ip(2)}) since $W(R')=W(q^2-1)/2$; specifically,
\[ \sqrt{q} \geq 2W(q^2-1) . \]
The above, with the help of Lemma~\ref{lem:w(r)}, implies that the case
\[ q \geq q_0 = (2\cdot 4514.7)^4 \simeq 6.65 \cdot 10^{15} \]
is settled. Next, let $t(q)$ stand for the number of prime factors of $q^2-1$. A quick computation reveals that, if $t(q) \geq 14$, then $q\geq q_0$, i.e.,   the case $t(q)\geq 14$ is settled.

Let $p(i)$ stand for the $i$-th prime (for example $p(2)=3$). Based on Proposition~\ref{prop:siev1}, we employ the following algorithm that takes $t_1\leq t_2$ as input and goes through the following steps:
\begin{algorithm}
\footnotesize
\caption{Settling the case $t_1\leq t(q)\leq t_2$.\label{alg:0}}
\begin{algorithmic}[1]
\State {\bf input:} integers $t_1\leq t_2$
\State {\bf output:} {\tt true} or {\tt false}
\Statex
\State $s\gets 0$ \Comment Step~1
\State $\varepsilon_1 \gets 1$
\While{$s\leq t_1$ {\bf and} $\varepsilon_1 - 1/p(t_1-s)>0$}
 \State $s\gets s+1$
 \State $\varepsilon_1 \gets \varepsilon_1 - 1/p(t_1-s+1)$
 \EndWhile
\Statex
\State $q_1 \gets \left( 2\cdot 2^{t_2-s} \cdot \left( \frac{s-1}{\varepsilon_1} + 2 \right) \right)^2$ \Comment Step~2
\Statex
\State $c\gets 1$ \Comment Step~3
\While{$p(1)\cdots p(c+1) \leq q_1^2-1$}
 \State $c\gets c+1$
 \EndWhile
\Statex
\If{$c\leq t_1$} \Comment Step~4
 \State {\bf return} {\tt true}
 \Else
 \State {\bf return} {\tt false}
 \EndIf
\end{algorithmic}
\end{algorithm}
%
% \begin{description}
%   \item[Step~1] Find the largest $s\leq t_1$ such that $\varepsilon_1:= 1-\sum_{i=0}^{s-1} 1/p(t_1-i) >0$.
%   \item[Step~2] Compute $q_1 := \left( 2\cdot 2^{t_2-s} \cdot \left( \frac{s-1}{\varepsilon_1} + 2 \right) \right)^2$.
%   \item[Step~3] Find the largest $c$ such that $p(1)\cdots p(c) \leq q_1^2-1$.
%   \item[Step~4] If $c\leq t_1$ return SUCCESS, otherwise return FAIL.
% \end{description}
If Algorithm~\ref{alg:0} returns {\tt true}, then the case $t_1\leq t(q)\leq t_2$ is settled.

Let us now explain the validity of Algorithm~\ref{alg:0}. Assume that the returned value is {\tt true} for some $t_1\leq t_2$. Take some $q$, such that $t_1\leq t(q)\leq t_2$ and write $q^2-1 = p_1^{s_1} \cdots p_{t(q)}^{s_{t(q)}}$, where the $p_i$'s are the (distinct) prime factors of $q^2-1$ in ascending order. It is clear that $W(q^2-1) = 2^{t(q)}$.  Thus a condition for our purposes, implied by Proposition~\ref{prop:siev1}, is
\begin{equation} \label{eq:cond_alg}
q \geq \left( 2\cdot 2^{t(q) -s}\cdot \left( \frac{s-1}{\varepsilon} + 2 \right) \right)^2 .
\end{equation}
Of course, $p_i\leq p(i)$, which impliess that $\varepsilon_1 \leq \varepsilon = 1-\sum_{i=0}^{s-1} 1/p_{t_1-i}$, and that $t(q)\leq t_2$, that is, the quantity $q_1$ computed in Step~2 is in fact larger than the right side of \eqref{eq:cond_alg}; hence, if $q\geq q_1$, then \eqref{eq:cond_alg} holds. The number $c$ in Step~3 stands for the maximum number of prime divisors a number not larger than $q_1^2-1$ can admit. Accordingly, if $c\leq t_1\leq t(q)$, then, \eqref{eq:cond_alg} holds, which is exactly the test that is performed in Step~4.

We successfully apply Algorithm~\ref{alg:0} for the pairs $(t_1,t_2)=(11,13)$ and $(10,10)$; consequently, the case $t(q)\geq 10$ is settled. Thus, we may now assume that $t(q)\leq 9$ and focus on the case
\[ q \leq (2\cdot 2^9)^2 = 1{,}048{,}576. \]

The interval $3\leq q\leq 1{,}048{,}576$ contains precisely $82{,}247$ odd prime powers. We first exploit  Proposition~\ref{prop:n=2(1)}. A quick computation reveals that, in the interval in question, there are exactly $2{,}425$ odd prime powers, where \eqref{eq:n=2_ip(1)} or \eqref{eq:n=2_ip(2)}, respectively, do not hold when all the relevant quantities are explicitly computed.Among these, $q=1{,}044{,}889$ is the largest.

We proceed to the sieving part, i.e., Proposition~\ref{prop:siev1}. Namely, we attempt to satisfy the conditions of Proposition~\ref{prop:siev1} as follows. Until we run out of prime divisors of $k$, or until $\varepsilon\leq 0$, we  add to the set of sieving primes (that is, the primes $p_1,\ldots ,p_s$ in Proposition~\ref{prop:siev1}) the largest prime divisor not already in the set. If, for one such set of sieving primes, the condition of Proposition~\ref{prop:siev1} is valid, then
% we remove the prime power in question from the list of potential exceptions.
the desired result holds for the prime power in question.

This procedure was successful, for all the $2{,}425$ prime powers mentioned earlier, with the $101$ exceptions of Table~\ref{tab:exc1}.
\begin{table}[h]
  \begin{center}
  \footnotesize
    \begin{tabular}{|p{0.85\textwidth}|c|}
      \hline $q$ & \# \\
      \hline 3, 5, 7, 9, 11, 13, 17, 19, 23, 25, 27, 29, 31, 37, 41, 43, 47, 49,
53, 59, 61, 67, 71, 73, 79, 81, 83, 89, 97, 101, 103, 109, 113, 121,
125, 127, 131, 137, 139, 149, 151, 157, 169, 173, 181, 191, 197, 199,
211, 229, 239, 241, 269, 281, 307, 311, 331, 337, 349, 361, 373, 379,
389, 409, 419, 421, 461, 463, 509, 521, 529, 569, 571, 601, 617, 631,
659, 661, 701, 761, 769, 841, 859, 881, 911, 1009, 1021, 1231, 1289,
1301, 1331, 1429, 1609, 1741, 1849, 1861, 2029, 2281, 2311, 2729, 3541 & 101 \\ \hline
    \end{tabular}
  \end{center}
  \caption{Odd prime powers that do not satisfy the conditions of Proposition~\ref{prop:siev1}.\label{tab:exc1}}
\end{table}
So, to sum up, we have proved the following.
\begin{theorem}\label{thm:main_n=r=2}
For every odd prime power $q$ not listed on Table~$\ref{tab:exc1}$, $\alpha\in\F_{q^n}^*$ and $\theta\in\F_{q^2}\setminus \F_q$, there exists some $x\in\F_q$ such that $\alpha(\theta+x)$ is $2$-primitive. In particular, $L_2(2) \leq 3541$.
\end{theorem}
We note that the above implies Theorems~\ref{thm:trans2} and \ref{thm:line} for any $q$ not present in Table~\ref{tab:exc1}. On what follows, we deal with these cases.
\section{Direct verification of the translate property}\label{subsec:exact}
With Theorem~\ref{thm:main_n=r=2} in mind, we move on to check the remaining cases, namely the prime powers listed in Table~\ref{tab:exc1}. First, we consider the translate property, i.e., we fix $\alpha=1$. Towards this end, we use Algorithm~\ref{alg:1}, which we implement in {\sc SageMath}.

\begin{algorithm}
\footnotesize
\caption{Explicitly verifying the translate property when $r=n=2$.\label{alg:1}}
\begin{algorithmic}[1]
\State {\bf input:} $q$ \Comment It has to be an odd prime power
\State {\bf output:} {\tt true} or {\tt false}
\Statex
\Procedure {NotInTranslate}{$j$, $A$, $q$, $a$} \label{procbe}
 \For {$i\in A$}
  \If {$(a^i-a^j)^{q-1}=1$}
   \State {\bf return} {\tt false}
  \EndIf
 \EndFor
 \State {\bf return} {\tt true}
\EndProcedure \label{procend}
\Statex
\State $a \gets$ a primitive element of $\F_{q^2}$
\State $A \gets \emptyset$
\Statex
\For{$j \gets 1, q^2-2$}
  \If{$\gcd(j,q^2-1)=2$ {\bf and} {\sc NotInTranslate}($j$, $A$, $q$, $a$)={\tt true}}\label{if}
   \State $A \gets A\cup \{j\}$
   \If {$|A|=q-1$} \label{if2}
    \State {\bf return} {\tt true}
    \EndIf
   \EndIf
 \EndFor
\State {\bf return} {\tt false} \label{return}
\end{algorithmic}
\end{algorithm}

We explain the validity and the ideas behind Algorithm~\ref{alg:1}. First, a primitive element $a\in\F_{q^2}$ is found; so we represent $\F_{q^2}^*$ as the powers of $a$ and work mostly with the (integer) exponents, rather than the finite field elements themselves. Then, starting with the {\bf if} statement in line~\ref{if}, we build the list $A$, representing the exponents of $a$ that correspond to $2$-primitive elements of $\F_{q^2}$ such that any two of them belong to a different set of translates. Note that an exponent has a $\gcd$ with $q^2-1$ equal to $2$ if and only if it corresponds to a $2$-primitive element. Also, Procedure~{\sc NotInTranslate} (lines~\ref{procbe}--\ref{procend}) checks whether a given exponent corresponds to some element whose set of translates has already been considered or not.

It follows that $|A|$ represents the number of set of translates of $\F_{q^2}/\F_q$ that include a $2$-primitive element. It is clear that the set $\F_{q^2}\setminus \F_q$ is, in fact, partitioned into the distinct sets of translates of $\F_{q^2}/\F_q$. Additionally, $|\F_{q^2}\setminus \F_q|=q(q-1)$ and, since every set of translates has cardinality $q$, it follows that there are exactly $q-1$ distinct sets of translates. Thus, $\F_{q^2}/\F_q$ has the translate property if and only if, at some point, $|A|$ reaches $q-1$. This is checked in line~\ref{if2}. On the contrary, if this number never reaches $q-1$, this extension does not have the translate property, see line~\ref{return}.

We ran Algorithm~\ref{alg:1} for all the 101 prime powers of Table~\ref{tab:exc1} and it returned {\tt true}, with the exception of $q=5$, $7$, $11$, $13$, $31$ and $41$. We note that for all these computations, a modern mid-range laptop spent about 2.5 hours of computer time. This completes the proof of Theorem~\ref{thm:trans2}.

\section{Direct verification of the line property}\label{subsec:conj}
We turn our attention to the line property. Fix some $\alpha\in\F_{q^2}^*$ and note that the lines of $\alpha$ and the various $\theta$'s over $\F_q$ define yet another partition of $\F_{q^2}\setminus\F_q$. For example, if $\alpha=1$ this partitioning coincides with the one that the sets of translates define.  This partitioning, however, is not unique to every $\alpha\in\F_{q^2}^*$, as we shall now demonstrate.

Let $\alpha_1,\alpha_2\in\F_{q^2}^*$ be such that $\alpha_1/\alpha_2=b_0\in\F_q$. It follows that an arbitrary line that $\alpha_1$ defines, along with some generator $\theta$ of the extension $\F_{q^2}/\F_q$, is of the form $\{ \alpha_1(\theta +x) \, : \, x\in\F_q \} = \{ \alpha_2(b_0\theta +b_0x) \, : \, x\in\F_q \} = \{ \alpha_2(b_0\theta +x) \, : \, x\in\F_q \}$, that is, one of the lines that $\alpha_2$ defines. Consequently, $\alpha_1$ and $\alpha_2$ are associated with the same partitioning.

Furthermore, set $A:=\{ \beta\in\F_{q^2}^* \, : \, \beta^{q+1}=1 \}$, $B:= \{ b_0\in\F_{q^2}^* \, : \, b_0^{q-1}=1 \} = \F_q^*$ and $C:= A B =  \{ \beta b_0 : \beta\in A \text{ and } b_0\in B \}$. By looking at the multiplicative orders, it is clear that $|A|= q+1$, $|B|= q-1$ and $A\cap B = \{ \pm 1 \}$. It follows that $|A B| = (q+1)(q-1)/2 = |\F_{q^2}^*|/2$. In addition, if we write $q^2-1=2^d q_0$, where $q_0$ is the odd part of $q^2-1$, it is clear that for odd $q$, $d\geq 3$ and that, if $\zeta$ is a primitive $2^d$-th root of unity, then $\zeta\not\in A B$, while $\zeta\beta b_0\not\in A B$ for any $\beta\in A$ and $b_0\in B$. In short, $\zeta AB \cap AB = \emptyset$ and $|\zeta AB| = |AB| = |F_{q^2}^*|/2$, thus $AB \cup \zeta AB = \F_{q^2}^*$.

Moreover, observe that $A=-A$. It follows that, since $q$ is odd, we may write $A = \{ \pm \beta_1 ,\ldots ,\pm \beta_\mu \}$, where $\mu=(q+1)/2$ and $\beta_i\neq \pm\beta_j$ for $i\neq j$. Now, set $A' := \{ \beta_1 ,\ldots, \beta_\mu\}$.  Evidently, $A'\cup (-A') = A$ and, from the fact that $-1\in B$, we deduce that $A'B \cup \zeta A'B = \F_{q^2}^*$.

From the above we conclude that, instead of the $q^2-1$ possible values of $\alpha$, in order to check the existence of a $2$-primitive element in every possible line, it suffices to check those lines corresponding to elements of the form $\gamma$, $\zeta\gamma$, where $\gamma\in A'$, i.e., the elements of $A'\cup \zeta A'$. Vitally, this reduces the number of possible values of $\alpha$ that require consideration to $q+1$.
\begin{algorithm}
\footnotesize
\caption{Explicitly verifying the line property when $r=n=2$.\label{alg:2}}
\begin{algorithmic}[1]
\State {\bf input:} $q$ \Comment It has to be an odd prime power
\State {\bf output:} {\tt true} or {\tt false}
\Statex
\Procedure {NotInLine}{$j$, $A$, $q$, $a$, $\gamma$} \label{procbe2}
 \For {$i\in A$}
  \If {$((a^i-a^j)/\gamma)^{q-1}=1$}
   \State {\bf return} {\tt false}
  \EndIf
 \EndFor
 \State {\bf return} {\tt true}
\EndProcedure \label{procend2}
\Statex
\Procedure {CheckLines}{$q$, $a$, $\gamma$}
  \State $A \gets \emptyset$
  \For{$j \gets 1, q^2-2$}
   \If{$\gcd(j,q^2-1)=2$ {\bf and} {\sc NotInLine}($j$, $A$, $q$, $a$, $\gamma$)={\tt true}}\label{if3}
     \State $A \gets A\cup \{j\}$
    \If {$|A|=q-1$} \label{if4}
      \State {\bf return} {\tt true}
      \EndIf
     \EndIf
 \EndFor
\State {\bf return} {\tt false} \label{return2}
\EndProcedure
\Statex
\State $a \gets$ a primitive element of $\F_{q^2}$
\State $q_0\gets$ the odd part of $q^2-1$ \Comment{write $q^2-1=2^d q_0$}
\State $\zeta\gets a^{q_0}$ \Comment{a primitive $2^d$-th root of unity}
\State $G \gets \emptyset$
\Statex
\For{$j \gets 1, \frac{q^2-1}{2}-1$}
  \If{$q-1\mid \gcd(j,q^2-1)$}
    \State $G\gets G\cup\{a^j,\zeta a^j\}$  
  \EndIf
\EndFor
\Statex
\For{$\gamma\in G$}
  \If{{\sc CheckLines}($q$, $a$, $\gamma$) $=$ {\tt false}}
    \State {\bf return} {\tt false}
  \EndIf
\EndFor
\State {\bf return} {\tt true}
\end{algorithmic}
\end{algorithm}

From  this observation, we use Algorithm~\ref{alg:2} which is based on Algorithm~\ref{alg:1}. Let us now explain its validity. The {\sc NotInLine} procedure is merely a generalization of the {\sc NotInTranslate} procedure of Algorithm~\ref{alg:1}, wherein the element $\gamma\in A'\cup \zeta A'$ is now considered. The procedure {\sc CheckLines} follows the same steps as the main procedure of Algorithm~\ref{alg:1}, with the difference that, instead of the sets of translates, the partition is now dictated by the lines that $\gamma$ defines. Note that for $\gamma=1$ the check that is performed in this step is identical to the one performed in Algorithm~\ref{alg:1}.
Finally, the main procedure of Algorithm~\ref{alg:2}, begins by building the set $G = A'\cup \zeta A'$. Since $a$ is primitive, $a^{j+ (q^2-1)/{2}} = -a^j$; so, in order to find a suitable $A'$, only the exponents $1,\ldots ,(q^2-1)/2$ need to be considered. After the set $G$ is built, the algorithm checks the output of {\sc CheckLines} for all $\gamma\in G$.

% To be more precise, after about 70 hours, the computer had answered for roughly half the prime powers of Table~\ref{tab:exc1}, in particular for $3\leq q\leq 239$. However, due to its recursive nature, the complexity of the program seems to be growing exponentially as $q$ grows, i.e., a complete solution within a reasonable time is not expected. On the other hand, based on these partial answers and Theorem~\ref{thm:main_n=r=2}, we have enough evidence safely to state the following conjecture:

We ran Algorithm~\ref{alg:2} for all the 101 prime powers of Table~\ref{tab:exc1} and it returned {\tt true}, with the exception of $q= 3$, $5$, $7$, $9$, $11$, $13$, $31$ and $41$. This completes the proof of Theorem~\ref{thm:line}.

\begin{remark}
By contrast to the residual  computation described in Section~\ref{subsec:exact} to establish Theorem~\ref{thm:trans2}, that   for completing the proof of Theorem~\ref{thm:line} turned out to be exceptionally expensive in terms of computer time. For example, $q=3541$ required 45 days of computer time, $q=2729$ required 20 days and $q=2029$ required 14 days, all on our mid-range modern laptop.  By way of comparison,   the computer  time consumed for smaller prime powers varied from 3--5 days, when $q\simeq 1000$ to a few seconds for $q\simeq 100$. The considerable cost for the larger numbers highlights the significance of strong theoretical methods that could minimize or, ideally, eliminate the computer dependency of our methods. For instance, a theoretical elimination of the two largest prime powers of Table~\ref{tab:exc1} would reduce the computer time spent by more than two months.
\end{remark}
\begin{remark}
As  the line property implies the translate property, the exceptional extensions appearing in Theorem~\ref{thm:trans2} also appear in Theorem~\ref{thm:line}. Unsurprisingly, the opposite is not true as the extensions $\F_{q^2}/\F_q$, for $q=3$ and $9$, possess the translate property but not the line property for $2$-primitive elements.
\end{remark}
%
%
%
%

%\bibliographystyle{abbrv}
%\bibliography{../../mybib}
%
%
\end{document}